\documentclass[12pt]{amsart}

\usepackage{amsthm,amssymb,amsfonts,amsmath,enumerate,graphicx,color}
\usepackage{mathrsfs}
\usepackage[margin=1in]{geometry}

\newcommand{\Ehr}[2]{\mathrm{Ehr}_{#1}(#2)}

\newcommand{\cP}{\mathcal{P}}

\newcommand{\Vol}{\operatorname{Vol}}

\newcommand{\Pyr}[1]{\operatorname{Pyr}(#1)}
 
\newcommand{\Z}{\mathbb{Z}}

\newcommand{\R}{\mathbb{R}}

\newcommand{\conv}{\mathrm{conv}}

\renewcommand{\phi}{\varphi}

\def\v{{\boldsymbol v}}
\def\x{{\boldsymbol x}}

\newcommand\commentout[1]{}

\newtheorem{theorem}{Theorem}

\newtheorem{proposition}[theorem]{Proposition}

\theoremstyle{remark}

\newtheorem{remark}[theorem]{Remark}
\theoremstyle{definition}

\begin{document}
\title[A self dual reflexive simplex]{self dual reflexive simplices with Eulerian polynomials}

\author{Takayuki Hibi}
\address{(Takayuki Hibi) Department of Pure and Applied Mathematics,\\ Graduate School of Information Science and Technology\\ Osaka University, Suita, Osaka 565-0871,Japan}
\email{hibi@math.sci.osaka-u.ac.jp}

\author{McCabe Olsen}
\address{(McCabe Olsen) Department of Mathematics\\
         University of Kentucky\\
         Lexington, KY 40506--0027}
\email{mccabe.olsen@uky.edu}

\author{Akiyoshi Tsuchiya}
\address{(Akiyoshi Tsuchiya) Department of Pure and Applied Mathematics,\\ Graduate School of Information Science and Technology\\ Osaka University, Suita, Osaka 565-0871,Japan}
\email{a-tsuchiya@cr.math.sci.osaka-u.ac.jp}

\thanks{ The authors would like to thank anonymous referees for reading the manuscript carefully and providing helpful comments and suggestions.
The second author was partially supported by a 2016 National Science Foundation/Japanese Society for the Promotion of Science East Asia and Pacific Summer Institutes Fellowship NSF OEIS--1613525.
The third author was partially supported by Grant-in-Aid for JSPS Fellows 16J01549.}




\begin{abstract}
A lattice polytope $\cP$ is called reflexive if its dual $\cP^\vee$ is a lattice polytope. The property that $\cP$ is unimodularly equivalent to $\cP^\vee$ does not hold in general, and in fact there are few examples of such polytopes. In this note, we introduce a new reflexive simplex $Q_n$ which has this property. Additionally, we show that $\delta$-polynomalial of $Q_n$ is the Eulerian polynomial and show the existence of a regular, flag, unimodular triangulation. 
\end{abstract}

\maketitle

Let $\cP\subset\R^d$ be a $d$-dimensional lattice polytope,
that is, a convex polytope all of whose vertices belong to $\Z^d$.
Let $\Vol(\cP)$ denote the \emph{normalized volume} of $\cP$, which is $d!$ times the Euclidean volume (Lebesgue measure) of $\cP$. For $k\in \Z_{>0}$, the \emph{lattice point enumerator} $i(\cP,k)$ counts the number of lattice points in $k\cP=\{k\alpha:\alpha\in \cP\}$, the  $k$th dilation of $\cP$. That is,
	\[
	i(\cP,k)=\#(k\cP\cap \Z^d), \ \ k\in\Z_{>0}.
	\]
Provided that $\cP$ is a lattice polytope, it is known  that $i(\cP,k)$ is a polynomial in the variable $k$ of degree $d$ (\cite{Ehrhart}).  
The \emph{Ehrhart Series} for $\cP$, $\Ehr{\cP}{z}$, is the rational generating function 
	\[
	\Ehr{\cP}z= 1+\sum_{k\geq 1}i(\cP,k)z^k=\frac{\delta(\cP,z)}{(1-z)^{d+1}}
	\]
where $\delta(\cP,z)=1+\delta_1z+\delta_2z^2+\cdots+\delta_dz^d$ is the \emph{$\delta$-polynomial} of $\cP$ (cf. \cite[Chapter 9]{ACCP}). The $\delta$-polynomial is endowed with the following properties:
	\begin{itemize}
	\item $\delta_0=1$, $\delta_1=i(\cP,1)-(d+1)$, and $\delta_d=\#(\cP\setminus \partial\cP\cap \Z^d)$;
	\item $\delta_i\geq 0$ for all $0\leq i\leq d$ (\cite{StanleyNonNeg});
	\item If $\delta_d\neq 0$, then $\delta_1\leq \delta_i$ for each $0\leq i\leq d-1$ (\cite{HibiBounds}). 
	\end{itemize}
For proofs of the first three properties of the coefficients, the reader should consult \cite[Chapter 9]{ACCP} or \cite[Chapter 3]{BeckRobins}.
The Ehrhart series and $\delta$-polynomials for polytopes have been studied extensively. For a detailed background on these topics, please refer to \cite{BeckRobins,Ehrhart,ACCP,EC1}.

Given two polytopes $\cP_1$ and $\cP_2$ in $\R^d$, we say that $\cP_1$ and $\cP_2$ are \emph{unimodularly equivalent} if there exists a unimodular matrix $U\in\Z^{d\times d}$ (i.e. $\det(U)=\pm 1$) and an integral vector $\v \in \Z^d$, such that $\cP_2=f_U(\cP_1)+\v$, where $f_U$ is the linear transformation defined by $U$, i.e., $f_U({\bf v}) = {\bf v} U$ for all ${\bf v} \in \R^d$. . We write $\cP_1\cong \cP_2$ in the case of unimodular equivalence. It is clear that if $\cP_1\cong\cP_2$, then $\delta(\cP_1,z)=\delta(\cP_2,z)$.

We say that a lattice polytope $\cP$ is \emph{reflexive} if the origin is the unique interior lattice point of $\cP$ and its dual polytope
	\[
	\cP^\vee=\left\{y\in\R^d \ : \langle x,y \rangle\leq 1 \mbox{ for all } x\in\cP \right\}
	\]
is a lattice polytope. 
Moreover, it follows from \cite{HibiDualPolytopes} that the following statements are equivalent:
	\begin{itemize}
	\item $\cP$ is unimodularly equivalent to some reflexive polytope;
	\item  $\delta(\cP,z)$ is of degree $d$ and is symmetric, that is $\delta_i=\delta_{d-i}$ for $0\leq i \leq \lfloor \frac{d}{2} \rfloor$.
	\end{itemize}


A polytope $\cP$ is called \emph{self dual} if $\cP$ is unimodularly equivalent to its dual polytope $\cP^\vee$. 
This is an extremely rare property in reflexive polytopes, especially for reflexive simplices. There are two families known self dual reflexive simplices. 
The first such family is given in \cite{BNill} and the  second family is given in \cite{TsuchiyaDeltaVector}. 
A construction for self dual reflexive polytopes is given in \cite{TsuchiyaDeltaVector}, though these polytopes are not simplicial and hence not simplices.   
In this paper, we provide a new family of self dual reflexive simplices $Q_n$ with small volume.


We now define a family of reflexive simplices which are self dual. 
For $n \geq 2$, let $Q_n$ denote the $n-1$ dimensional simplex with $\mathcal{V}$-representation 

	\[
	Q_n:=\conv \begin{bmatrix}
	1 & 1-n & 0 & 0 & \cdots & 0\\
	1 & 1 & 2-n & 0 & \cdots & 0\\
	1 & 1 & 1 & 3-n & \cdots & 0\\
	\vdots & \vdots & \vdots & & \ddots & \vdots\\
	1 & 1 & 1 & \cdots & 1 & -1\\
	\end{bmatrix}
	\]
	where we use the convention that the $x_{n-1}$ coordinate is given by the first row and the $x_1$ is given by the last row. 
We will adopt for simplicity the notation $C_{i}$ where $i\in\{0,\cdots,n-1\}$ for each column (vertex) such that $Q_n=\conv[C_{0} \ C_{1} \ \cdots \ C_{n-1}]$.

We have the following theorem.
\begin{theorem}\label{selfdual}
For $n\geq 2$, we have $Q_n\cong Q_n^\vee$. 
\end{theorem}



It behooves us to introduce an $\mathcal{H}$-representation for the simplex to compute its dual polytope. We now give such a representation. 

\begin{proposition}\label{hrepQn}
For $n \geq 2$, $Q_n$ has the $\mathcal{H}$-representation
	\[
	Q_n=\left. \begin{cases}
	 x\in \R^{n-1} \ : & kx_k-\sum_{i=1}^{k-1}x_i\leq 1 \ , \  \ 1\leq k \leq n-1\\
	   & -\sum_{i=1}^{n-1}x_i\leq 1\\
	\end{cases} \right\}.
	\]
\end{proposition} 
\begin{proof}
It is sufficient to show that the vertices of $Q_n$ each satisfy precisely $n-1$ of the halfspace inequalities with equality and  satisfies the other inequalty strictly. Let  $f_k(\x)=kx_k-\sum_{i=1}^{k-1}x_i$,  and $f_{n}(\x)=-\sum_{i=1}^{n-1}x_i$. 
For a vertex $C_j$, we have that  $f_k(C_j)= 1$ for all $k\neq n-j$. 
This follows, because if $k<n-j$, we have $f_k(C_j)=(k)(1)-\sum_{i=1}^{k-1}1=1$, if $k>n-j$ with $k\neq n$, we have $f_k(C_j)=(n-j)-\sum_{i=1}^{n-j-1}1=1$, and if $k=n>n-j$, we have $f_n(C_j)= -(j-n)-\sum_{i=1}^{n-j-1}1=1$.
In the case of $k=n-j$, $f_{n-j}(C_j)=-(n-j)^2-(n-1-j)<1$ if $j\neq 0$ and $j\neq n-1$. For $j=0$ we have $f_n(C_0)=1-n<1$ and for $j=n-1$, we have $f_1(C_{n-1})=-1<1$. 
Thus, we have the correct $\mathcal{H}$-representation. 
\end{proof}
 
By \cite[Corollary 35.3]{ACCP},  and Proposition \ref{hrepQn}, 
it is clear that  $Q_n^\vee =-Q_n$.
Therefore, we have shown Theorem \ref{selfdual}.
 

\begin{remark}
We should note that $\Vol(Q_n)=n!$. For $n\geq 4$, it is immediate that these polytopes are different than previously known self dual reflexive simplices given in \cite{BNill,TsuchiyaDeltaVector}. 
\end{remark}	

Moreover, the self dual reflexive simplex of $Q_n$ has an interesting $\delta$-polynomial and a special triangulation.

\begin{theorem}\label{ehrhartseries}
	Let $n \geq 2$.\\
\textnormal{(i)} We have $\delta(Q_n,z)=A_n(z),$
where $A_n(z)$ is the Eulerian polynomial.\\
\textnormal{(ii)} $Q_n$ has a regular, flag, unimodular triangulation. 
	
\end{theorem}
\begin{proof}
	For a lattice polytope $\cP \subset \R^d$, we set $\Pyr{\cP}=\conv(\cP\times\{0\}, (0,\cdots,0,1)) \subset \R^{d+1}$.
	Then it is well-known that $\delta(\Pyr{\cP},z)=\delta(\cP,z)$ (cf. \cite[Section 2.4]{BeckRobins})
	and $\cP$ has a regular, flag, unimodular triangulation if and only if $\Pyr{\cP}$ has a regular, flag, unimodular triangulation (cf. \cite[Section 4.2]{Triangulations}). 
	
	Let $R_n$ denote the $n$ dimensional simplex with $\mathcal{V}$-representation 
		\[
	R_n:=\conv \begin{bmatrix}
	0 & n & n & n & \cdots & n\\
	0 & 0 & n-1 & n-1 & \cdots & n-1\\
	0 & 0 & 0 & n-2 & \cdots & n-2\\
	\vdots & \vdots & \vdots & & \ddots & \vdots\\
	0 & 0 & 0 & \cdots & 0 & 1\\
	\end{bmatrix}
	\subset \R^n.
	\]
	This polytope $R_n$ is called a \emph{lecture hall polytope}.
Notice that $\Pyr{Q_n}$ is unimodularly equivalent to $R_n$. Let $\widetilde{R_n}$ be the polytope defined from $R_n$ by removing the $(n+1)$th column and $n$th row, let $\mathcal{U}_{n}$ denote the $(n-1)\times(n-1)$ upper triangular matrix defined by $(\mathcal{U}_{n})_{ij}=1$ if $i\leq j$ and $(\mathcal{U}_{n})_{ij}=0$ otherwise, let $\mathbf{1}$ denote the all ones vector, and let $\mathbf{0}$ denote the all zeros vector. Then we have   
	\[
	Q_n\cong -f_{\mathcal{U}_{n}}\left( Q_n-\mathbf{1}\right)=\widetilde{R_n}.
	\]
Hence it follows that 
	\[
	\Pyr{Q_n}\cong \operatorname{Pyr}\left({\widetilde{R_n}}\right)\cong 
	 R_n.
	\]
It is known that  for $n \geq 2$, $\delta(R_n,z)=A_n(z)$ (\cite{SavageSchuster})
and $R_n$ has a regular, flag, unimodular triangulation (\cite{BBKSZ}).
Therefore, the assertion follows.
\end{proof}

\end{document}